\documentclass[a4paper,11pt]{article}
\usepackage{amsmath,amsthm,amssymb}
\usepackage{enumerate}

\usepackage[english]{babel}
\usepackage[utf8]{inputenc}
\usepackage[a4paper]{geometry}
\usepackage{latexsym}
\usepackage{amscd}
\usepackage{graphics}
\usepackage{color}
\usepackage{array}
\usepackage{mathrsfs}
\usepackage{graphicx}
\usepackage{stmaryrd}
\usepackage{mathabx}
\usepackage{dsfont}
\usepackage{comment}
\usepackage{calligra}
\usepackage{varioref}
\usepackage{prettyref}

\def\subjclass#1{{\renewcommand{\thefootnote}{}%
\footnote{\emph{Mathematics Subject Classification (2010):} #1}}}
\def\keywords#1{{\renewcommand{\thefootnote}{}%
\footnote{\emph{Keywords:} #1}}}
\def\ackn#1{{\renewcommand{\thefootnote}{}%
\footnote{#1}}}

\newtheorem{thm}{Theorem}[section]

\newtheorem{lem}[thm]{Lemma}

\newtheorem{prop}[thm]{Proposition}

\newtheorem{question}[thm]{Question}

\newrefformat{sec}{section \ref{#1}}
\newrefformat{thm}{Theorem \ref{#1}}
\newrefformat{cor}{Corollary \ref{#1}}
\newrefformat{lem}{Lemma \ref{#1}}
\newrefformat{prob}{Problem \ref{#1}}
\newrefformat{prop}{Proposition \ref{#1}}
\newrefformat{conj}{Conjecture \ref{#1}}
\newrefformat{fact}{Fact \ref{#1}}
\newrefformat{question}{Question \ref{#1}}

\theoremstyle{definition}
\newtheorem{defin}[thm]{Definition}
\newtheorem{rem}[thm]{Remark}

\numberwithin{equation}{section}

\newcommand{\Z}{\mathbb{Z}}
\newcommand{\K}{\mathbb{K}}
\newcommand{\C}{\mathbb{C}}
\newcommand{\N}{\mathbb{N}}
\newcommand{\R}{\mathbb{R}}

\DeclareMathOperator{\Id}{Id}
\DeclareMathOperator{\nul}{n}
\DeclareMathOperator{\dfc}{d}
\DeclareMathOperator{\ind}{ind}
\DeclareMathOperator{\incl}{i}
\DeclareMathOperator{\Emb}{Emb}
\DeclareMathOperator{\GL}{GL}

\makeatletter
\newcommand{\alaligne}{~\vspace*{\topsep}\nobreak\@afterheading}
\makeatother

\usepackage{glossaries}

\newacronym{hi}{HI}{hereditarily indecomposable}
\newcommand{\cL}{\mathcal{L}}
\newcommand{\Fred}{\mathcal{F}red}
\newcommand{\SFred}{\hat{\mathcal{F}}red}

\begin{document}






\title{Spectral-free methods in the theory of hereditarily indecomposable Banach spaces}

\author{N. de Rancourt}

\date{}

\maketitle


\subjclass{46B03, 46B04, 46B20, 47A53}

\ackn{The author was partially supported by the joint FWF–GA\v{C}R Grant No. 17-33849L: Filters, Ultrafilters and Connections with Forcing}

\keywords{Hereditarily indecomposable Banach spaces, Fredholm theory}


\begin{abstract}

We give new and simple proofs of some classical properties of hereditarily indecomposable Banach spaces, including the result by W. T. Gowers and B. Maurey that a hereditarily indecomposable Banach space cannot be isomorphic to a proper subspace of itself. These proofs do not make use of spectral theory and therefore, they work in real spaces as well as in complex spaces. We use our method to prove some new results. For example, we give a quantitative version of the latter result by Gowers and Maurey and deduce that Banach spaces that are isometric to all of their subspaces should have an unconditional basis with unconditional constant arbitrarily close to $1$. We also study the homotopy relation between into isomorphisms in hereditarily indecomposable spaces.

\end{abstract}

\section{Introduction}

In this paper, unless otherwise specified, when speaking about a \textit{Banach space} (or simply a \textit{space}), we shall mean an infinite-dimensional Banach space, and by \textit{subspace} of a Banach space, we shall always mean infinite-dimensional, closed subspace. By \textit{operator}, we shall always mean bounded linear operator, and by \textit{isometry}, we shall mean linear (non-necessarily sujective) isometry.

\smallskip

In 1993, W. T. Gowers and B. Maurey \cite{GowersMaureyHI} built the first example of a Banach space containing no unconditional basic sequence, thus solving the longstanding \textit{unconditional basic sequence problem}. The space they built actually has a much stronger property: it is \textit{\gls{hi}}, i.e. no two subspaces of it are in topological direct sum. In the same paper, the authors prove several properties of \gls{hi} spaces. Among them, the following:

\begin{thm}\label{thm:main}
An \gls{hi} space is isomorphic to no proper subspace of itself.
\end{thm}

Following Maurey \cite{Maurey}, an operator $T\colon X \to Y$ between two Banach spaces will be called an \textit{into isomorphism} if it is an isomorphism between $X$ and its image (or equivalently, if there exists $c > 0$ such that $\forall x \in X \; \|T(x)\| \geqslant c\|x\|$). Recall that an operator $S\colon X \to Y$ between two Banach spaces is \textit{strictly singular} if no restriction of $S$ to a subspace of $X$ is an into isomorphism. In their paper, Gowers and Maurey get \prettyref{thm:main} for complex \gls{hi} spaces as an immediate consequence (using basic Fredholm theory) of the following result:

\begin{thm}\label{thm:operatorsHI}
If $X$ is a complex \gls{hi} space then every operator $T\colon X \to X$ can be written $T = \lambda\Id_X + S$, where $\lambda \in \C$ and $S$ is a strictly singular operator.
\end{thm}

Their proof of \prettyref{thm:operatorsHI} makes use of Fredholm theory and spectral theory, thus explaining the fact that it only works for complex spaces. Actually, it turns out that this result is not true in general for real \gls{hi} spaces. For such spaces $X$, the authors prove a general structural result for operators from $X$ to $X$, similar to \prettyref{thm:operatorsHI}, by passing to the complexification of $X$. From this result, they deduce \prettyref{thm:main} for real \gls{hi} spaces. As they mention, they do not know any direct proof of \prettyref{thm:main} in the real case. Later, V. Ferenczi \cite{FerencziHI} gave a new proof of \prettyref{thm:operatorsHI} using no spectral theory but rather Banach algebra methods (he actually proves a more general result); however, as far as the author of the present article knows, no direct proof of \prettyref{thm:main} for real spaces has been known by now.

\smallskip

The main aim of this paper is to introduce new methods in the theory of \gls{hi} spaces, using no spectral theory but only Fredholm theory, and thus working in real \gls{hi} spaces as well as in complex \gls{hi} spaces. These methods will allow us to give new and simpler proofs of some classical properties of \gls{hi} spaces, including \prettyref{thm:main}; this will be done in \prettyref{sec:ProofMain}. Our hope is that these new methods can be adapted more easily to broader contexts. We present such an adaptation in \prettyref{sec:Isomogeneous}, where a quantitative version of \prettyref{thm:main} is proved. This allows us to start an investigation of the \textit{isometrically homogeneous space problem}, asking if a space that is isometric to all of its subspaces should necessarily be isometric to $\ell_2$. Finally, in \prettyref{sec:Homotopy}, we prove some results about the homotopy relation between into isomorphisms from \gls{hi} spaces. We deduce from them that the general linear group of the real \gls{hi} space built by Gowers and Maurey in \cite{GowersMaureyHI} has exactly $4$ connected components.

\smallskip

We start this paper with recalling some results in Fredholm theory, in \prettyref{sec:Fredholm}.


\bigskip

\section{Fredholm theory}\label{sec:Fredholm}

In this section, we recall some results in Fredholm theory that will be useful in the rest of this paper.  For a more detailed presentation and proofs of these results, we refer to the survey \cite{Maurey} on Banach spaces with few operators, that contains a good introduction to Fredholm theory (see sections 3, 4 and 6).

\smallskip

In the rest of this paper, all Banach spaces will be over the field $\K:= \R$ or $\C$, the proofs working as well in both cases.

\smallskip

Let $X, Y$ be two Banach spaces. Let $\cL(X, Y)$ denote the space of bounded operators from $X$ to $Y$ (when $X = Y$, then this space will simply be denoted by $\cL(X)$). Equip $\cL(X, Y)$ with the operator norm denoted by $\|\cdot\|$, and with the associated topology. For $T\in \cL(X, Y)$, let $\nul(T) \in \N \cup \{+\infty\}$ and $\dfc(T)\in \N \cup \{+\infty\}$ denote respectively the dimension of the kernel of $T$ and the codimension of the image of $T$. The operator $T$ is said to be \textit{semi-Fredholm} if it has closed image and if one of the numbers $\nul(T)$ and $\dfc(T)$ is finite, and \textit{Fredholm} if both numbers $\nul(T)$ and $\dfc(T)$ are finite (this implies that $T$ has closed image). Let $\Fred(X, Y)$ and $\SFred(X, Y)$ denote respectively the set of Fredholm operators and of semi-Fredholm operators from $X$ to $Y$, seen as subsets of $\cL(X, Y)$ with the induced topology. For $T \in \SFred(X, Y)$, the \textit{Fredholm index} of $T$ is defined by $\ind(T) = \nul(T)-\dfc(T) \in \Z \cup\{- \infty, + \infty\}$. One of the fundamental results of Fredholm theory is the following:

\begin{thm}\label{thm:ContFredholm}
The Fredholm index $\ind : \SFred(X, Y) \longrightarrow \mathbb{Z} \cup \{- \infty, + \infty\}$ is locally constant.
\end{thm}

Say that an operator $T \colon X \to Y$ is \textit{finitely singular} if there exists a finite-codimensional subspace $X_0$ of $X$ such that $T_{\restriction X_0}$ is an into isomorphism. An operator which is not finitely singular will be called \textit{infinitely singular}. We have the two following results:

\begin{prop}\label{prop:CaracFinSing}

An operator $T \colon X \to Y$ is finitely singular if and only if it is semi-Fredholm with $\ind(T) < + \infty$.

\end{prop}

\begin{prop}\label{prop:CaracInfSing}

An operator $T \colon X \to Y$ is infinitely singular if and only if for every $\varepsilon > 0$, there exists a subspace $Y$ of $X$ such that $\left\|T_{\restriction Y}\right\| \leq \varepsilon$.

\end{prop}

Finally we mention a last result that we will not use in our proofs, but which is the key of the originial proof of \prettyref{thm:main} from \prettyref{thm:operatorsHI}:

\begin{prop}\label{prop:SSingPerturbation}
Let $T \colon X \to Y$ be a Fredholm operator and $S \colon X \to Y$ be a strictly singular operator. Then $T + S$ is Fredholm and $\ind(T + S) = \ind(T)$.
\end{prop}

The latter result, combined with \prettyref{thm:operatorsHI}, gives that every element of $\cL(X)$, for $X$ a complex \gls{hi} space, is either strictly singular or Fredholm with index $0$. In particular, such an operator cannot be an isomorphism between $X$ and a proper subspace.

\bigskip

\section{Subspaces, quotients and operators}\label{sec:ProofMain}

In this section we give a new proof of \prettyref{thm:main} and of some related results. Some elements of our proofs were already present in \cite{GowersMaureyHI}; to make this paper self-contained, we shall state them as lemmas and prove them.

\smallskip

Given $X$ a Banach space and $Y \subseteq X$ a subspace, the inclusion embedding $Y \hookrightarrow X$ will be denoted by $\incl_{Y, X}$ in the rest of this paper. We say that $\lambda \in \K$ is an \textit{infinitely singular value} of the operator $T \colon Y \to X$ when $T - \lambda\incl_{Y, X}$ is infinitely singular. The next lemma was already present as a remark in \cite{GowersMaureyHI}.

\begin{lem}\label{lem:OneInfSing}
If $X$ is a Banach space, $Y \subseteq X$ an \gls{hi} subspace, and $T \colon Y \to X$ an operator, then $T$ has at most one infinitely singular value.
\end{lem}
\begin{proof}
Suppose that $\lambda, \mu \in \K$ are two infinitely singular values of $T$. Let $\varepsilon > 0$; by \prettyref{prop:CaracInfSing} we can find two subspaces $U_\varepsilon, V_\varepsilon$ of $Y$ such that $\left\|(T - \lambda\incl_{Y, X})_{\restriction U_\varepsilon}\right\| \leq \varepsilon$ and $\left\|(T - \mu\incl_{Y, X})_{\restriction V_\varepsilon}\right\| \leq \varepsilon$. The subspaces $U_\varepsilon$ and $V_\varepsilon$ are not in topological direct sum, so we can find $u_\varepsilon \in S_{U_\varepsilon}$ and $v_\varepsilon \in S_{V_\varepsilon}$ such that $\left\|u_\varepsilon - v_\varepsilon\right\| \leq \varepsilon$. And we have:
\begin{eqnarray*}
|\lambda - \mu| & = & \left\| \lambda u_\varepsilon - \mu u_\varepsilon\right\| \\
& \leq & \left\|\lambda u_\varepsilon - \mu v_\varepsilon\right\| + |\mu|\cdot\left\|v_\varepsilon - u_\varepsilon\right\|\\
& \leq & \left\|\lambda u_\varepsilon - T(u_\varepsilon)\right\| + \left\|T(u_\varepsilon) - T(v_\varepsilon)\right\| + \left\|T(v_\varepsilon) - \mu v_\varepsilon\right\| + |\mu|\varepsilon\\
& \leq & \varepsilon + \|T\| \cdot \left\|u_\varepsilon - v_\varepsilon\right\| + \varepsilon + |\mu|\varepsilon\\
& \leq & (2 + |\mu| +\|T\|)\varepsilon.
\end{eqnarray*}
Thus, making $\varepsilon$ tend to $0$, we deduce that $\lambda = \mu$.
\end{proof}

The next lemma is the key of our proof. It is valid for every Banach space $X$, not necessarily \gls{hi}.

\begin{lem}\label{lem:main}
Let $X$ be a Banach space and $T \in \cL(X)$ be a semi-Fredholm operator with nonzero index. Then $T$ has at least two distinct real infinitely singular values, a positive one and a negative one.
\end{lem}
\begin{proof}
For $t \in [0, 1]$, define $T_t = tT + (1-t) \Id_X$. We show that there exists $t \in (0, 1)$ such that $T_t$ is infinitely singular; this will imply that $\frac{t-1}{t}$ is a negative infinitely singular value of $T$. Suppose not. Then by \prettyref{prop:CaracFinSing}, for every $t \in [0, 1]$, $T_t$ is semi-Fredholm. So letting $f(t) = \ind(T_t)$, we define a function $f : [0, 1] \longrightarrow \Z \cup \{- \infty, + \infty\}$; by \prettyref{thm:ContFredholm}, this function is locally constant, so constant. This is a contradiction since $f(0) = 0$ and $f(1) \neq 0$.

\smallskip

We prove in the same way that $T$ has a positive infinitely singular value, considering the operators $T_t' = tT - (1-t) \Id_X$.
\end{proof}

We can now conclude the proof of \prettyref{thm:main}. Let $X$ be an \gls{hi} space and $T\in\cL(X)$ be an isomorphism from $X$ to a proper subspace of $X$. Then $T$ is semi-Fredholm with $\ind(T) < 0$, so by \prettyref{lem:main} it has at least two distinct infinitely singular values. This contradicts \prettyref{lem:OneInfSing}.

\smallskip

We now give a new proof of a generalization of \prettyref{thm:main} first proved by Ferenczi \cite{FerencziQuotientHI}.

\begin{thm}
Let $X$ be an \gls{hi} space, and let $Z \subseteq Y$ be subspaces of $X$, with $Z$ not necessarily infinite-dimensional. Suppose that either the inclusion $Y \subseteq X$ is strict, or $Z\neq\{0\}$. Then $X$ is not isomorphic to $Y/Z$.
\end{thm}

\begin{proof}
We follow the same approach as for \prettyref{thm:main}. Let $\pi \colon Y \twoheadrightarrow Y/Z$ be the quotient map. Suppose that there exists an isomorphism $T \colon Y/Z \to X$. Then both $\incl_{Y, X}$ and $T \circ \pi$ are semi-Fredholm operators $Y \to X$, and we have $\ind(\incl_{Y, X}) \leqslant 0$ and $\ind(T \circ \pi) \geqslant 0$. Moreover, one of the latter inequalities has to be strict. So $\ind(\incl_{Y, X}) \neq \ind(T \circ \pi)$.

\smallskip

By the continuity of the Fredholm index, there should exist $t \in (0, 1)$ such that $tT \circ \pi + (1-t)\incl_{Y, X}$ is not semi-Fredholm. This means that there exists $\lambda < 0$ such that $T \circ \pi - \lambda\incl_{Y, X}$ is not semi-Fredholm, so infinitely singular by \prettyref{prop:CaracFinSing}. In the same way, considering the operators $tT \circ \pi - (1-t)\incl_{Y, X}$, we see that there is a $\mu > 0$ such that $T \circ \pi - \mu\incl_{Y, X}$ is infinitely singular. This contradicts \prettyref{lem:OneInfSing}.
\end{proof}

We finish this section with a direct proof of the following result, which, as said before, is an immediate consequence of \prettyref{thm:operatorsHI} and \prettyref{prop:SSingPerturbation} in the special case of complex spaces, but which is also valid for real spaces.

\begin{thm}\label{thm:operators}
Let $T\in \cL(X)$, where $X$ is an \gls{hi} space. Then either $T$ is Fredholm with index $0$, or $T$ is strictly singular.
\end{thm}

\smallskip

In order to prove \prettyref{thm:operators}, we need the following lemma, that was already present as a remark in \cite{GowersMaureyHI}.

\begin{lem}\label{lem:CaracStrictSing}
Let $T\in \cL(X)$, where $X$ is an \gls{hi} space. Then $T$ is infinitely singular if and only if it is strictly singular.
\end{lem}
\begin{proof}
One implication directly follows from the definitions. For the other, suppose that $T$ is infinitely singular. We fix $Y$ a subspace of $X$ and $\varepsilon > 0$, and we want to find $y \in S_Y$ such that $\|T(y)\| < \varepsilon$.

\smallskip

Since $T$ is infinitely singular, there exists a subspace $Z$ of $X$ such that $\left\|T_{\restriction Z}\right\| \leq \frac\varepsilon 2$. Since $Y$ and $Z$ are not in topological direct sum, we can find $y \in S_Y$ and $z\in S_Z$ with $\|y-z\| \leq \frac\varepsilon {2\|T\|}$. Then we have:
\[\|T(y)\| \leq \|T(z)\| + \|T(y - z)\| \leq \frac \varepsilon 2 + \|T\|\cdot \frac\varepsilon {2\|T\|} \leq \varepsilon,\]
as wanted.
\end{proof}

\begin{proof}[Proof of \prettyref{thm:operators}]
Suppose that $T$ is not strictly singular. By \prettyref{lem:CaracStrictSing}, it is finitely singular, so by \prettyref{prop:CaracFinSing}, it is semi-Fredholm. If it does not have index $0$, then by \prettyref{lem:main}, it has at least two infinitely singular values, contradicting \prettyref{lem:OneInfSing}.
\end{proof}

\bigskip\bigskip

\section{A quantitative version}\label{sec:Isomogeneous}

Say that a Banach space $X$ is \textit{homogeneous} is it is isomorphic to all of its subspaces, and \textit{isometrically homogeneous} if it is isometric to all of its subspaces. In \cite{CDDK}, the authors ask the following question:

\begin{question}[C\'uth--Dole\v{z}al--Doucha--Kurka]

Is every isometrically homogeneous Banach space isometric to $\ell_2$?

\end{question}

This problem, that we will call the \textit{isometrically homogeneous space problem}, is an isometric version of the \textit{homogeneous space problem}, asking whether every homogeneous space is isomorphic to $\ell_2$. The latter problem was solved positively in the nineties as a combination of the three followings results:
\begin{itemize}
    \item every Banach space either has a subspace isomorphic to $\ell_2$, or has a subspace without an unconditional basis (Komorowski--Tomczak-Jaegermann, \cite{KomorowskiTomczakJaegermann});
    \item every Banach space either has a subspace with an unconditional basis, or has an \gls{hi} subspace (\textit{Gowers' first dichotomy}, \cite{GowersRamsey});
    \item and \prettyref{thm:main}, by Gowers and Maurey.
\end{itemize}
To investigate the isometrically homogeneous space problem, it is tempting to try to find quantitative versions of the results above. Gowers' first dichotomy admits a quantitative version, that will be presented below (\prettyref{thm:Gowers1stDichotoQuant}) and which was proved by Gowers himself as a step of the proof of the same non-quantitative result. In this section, we give a quantitative version of \prettyref{thm:main}. This version, combined with the quantitative version of Gowers' first dichotomy, will give us the following result:

\begin{thm}\label{thm:isomogeneous}
Let $X$ be an isometrically homogeneous Banach space. Then for every $\varepsilon>0$, $X$ admits a $(1+\varepsilon)$-unconditional basis.
\end{thm}

We don't know if there is a sufficiently good quantitative version of Komorowski--Tomczak-Jaegermann's theorem to solve positively the isometrically homogeneous space problem.

\smallskip

In the rest of this section, we fix a Banach space $X$. Observe that if $Y$ and $Z$ are two subspaces of $X$, then $Y$ and $Z$ are not in topological direct sum if and only if for every $C \geqslant 1$, there exists $y \in Y$ and $z \in Z$ with $\|y+z\| > C \|y-z\|$. This motivates the following definition:

\begin{defin}

For $C \geqslant 1$, we say that $X$ is \textit{$C$-\gls{hi}} if for every subspaces $Y, Z \subseteq X$, there exists $y \in Y$ and $z \in Z$ such that $\|y+z\| > C \|y-z\|$.

\end{defin}

In particular, $X$ is \gls{hi} if and only if it is $C$-HI for arbitrarily large $C$. A slightly different definition of $C$-\gls{hi} spaces had already been given by Gowers in \cite{GowersRamsey}, for spaces with a basis, only taking into account block-subspaces. However, is is easy to see that these two notions are equivalent up to a constant.

\smallskip

Recall that two Banach spaces $Y$ and $Z$ are said to be \textit{$C$-isomorphic}, for some $C \geqslant 1$, if there is an isomorphism $T \colon Y \to Z$ with $\|T\|\cdot \|T^{-1}\| \leqslant C$. In particular, two spaces are $1$-isomorphic if and only if they are isometric. Our quantitative version of \prettyref{thm:main} is the following:

\begin{thm}\label{thm:mainQuant}

Let $C \geqslant 1$. Suppose that $X$ is $(C + \varepsilon)$-\gls{hi} for some $\varepsilon > 0$. Then $X$ is $C$-isomorphic to no proper subspace of itself.

\end{thm}

\begin{proof}

Suppose that $X$ is $C$-isomorphic to a proper subspace $Y \subseteq X$. Let $T \colon X \to Y$ be an isomorphism with $\|T\| \leqslant C$ and $\|T^{-1}\| = 1$; we see $T$ as an operator $X \to X$. By \prettyref{lem:main}, $T$ has at least two real infinitely singular values, a positive one that we denote by $\lambda$ and a negative one that we denote by $-\mu$. Replacing $T$ with $-T$ if necessary, we can assume that $\lambda \leqslant \mu$. We let $\nu = \mu - \lambda$. For every $a < 1$ and every  $x \in X$, we have $\|T(x) - a x\| \geqslant \|T(x)\| - a\|x\| \geqslant (1 - a)\|x\|$, so $T - a \Id_X$ is an into isomorphism and in particular, it is finitely singluar. We deduce that $\lambda \geqslant 1$.



\smallskip

We now use a similar method as in the proof of \prettyref{lem:OneInfSing} to conclude. Fix $\varepsilon > 0$. Then there exists subspaces $Y, Z \subseteq X$ such that $\left\|(T-\lambda\Id_X)_{\restriction Y}\right\| \leq \varepsilon$ and $\left\|(T+\mu\Id_X)_{\restriction Z}\right\| \leq \varepsilon$. Thus, for every $y \in Y$ and $z \in Z$, we have:
\begin{eqnarray*}
(\lambda + \mu)\|y+z\| & = & \|2 \lambda y + 2 \mu z + (\mu-\lambda)(y - z)\| \\
& \leqslant & 2\|\lambda y + \mu z\| + \nu \|y - z\| \\
& \leqslant & 2(\|\lambda y - T(y)\| + \|T(y) - T(z)\| + \|T(z) + \mu z\|) + \nu\|y - z\|\\
& \leqslant & 2\varepsilon(\|y\| + \|z\|) + 2C\|y - z\| + \nu \|y - z\| \\
&\leqslant & 2\varepsilon(\|y + z\| + \|y - z\|) + (2C + \nu)\|y - z\|
\end{eqnarray*}
Since $\lambda \geqslant 1$, we have $2 + \nu \leqslant \lambda + \mu$. Using this inequality on the left-hand side and the fact that $2C + \nu \leqslant (2 + \nu)C$ on the right-hand side, we get that:
$$(2 + \nu - 2\varepsilon)\|y+z\| \leqslant ((2 + \nu)C + 2\varepsilon)\|y - z\|.$$
We deduce that $X$ is not $\frac{(2 + \nu)C + 2\varepsilon}{2 + \nu - 2\varepsilon}$-\gls{hi}. Since $\frac{(2 + \nu)C + 2\varepsilon}{2 + \nu - 2\varepsilon}$ tends to $C$ when $\varepsilon$ tends to $0$, we get the wanted result.

\end{proof}

Now the quantitative version of Gowers' first dichotomy, proved in \cite{GowersRamsey} (see Corollary 3.2) is the following:

\begin{thm}[Gowers]\label{thm:Gowers1stDichotoQuant}

Let $C \geqslant 1$ and $\varepsilon > 0$. Then either $X$ has a subspace with a $(C + \varepsilon)$-unconditional basis, or a $C$-\gls{hi} subspace.

\end{thm}

This, combined with \prettyref{thm:mainQuant}, is enough to prove \prettyref{thm:isomogeneous}. Indeed, if $X$ is isometrically homogeneous, then, fixing $\varepsilon > 0$ and applying \prettyref{thm:mainQuant} to $C = 1$, we get that $X$ is not $(1+\varepsilon)$-\gls{hi}, so contains no $(1+\varepsilon)$-\gls{hi} subspace. By \prettyref{thm:Gowers1stDichotoQuant} applied to $C = 1 + \varepsilon$, we get that $X$ has a subspace with a $(1+2\varepsilon)$-unconditional basis, so $X$ itself has a $(1+2\varepsilon)$-unconditional basis.

\bigskip\bigskip

\section{Homotopy between into isomorphisms}\label{sec:Homotopy}

Given $X$ and $Y$ two Banach spaces, let $\Emb(Y, X)$ the set all $T \in \cL(Y, X)$ that are into isomorphisms. This is an open subset of $\cL(Y, X)$. In the case where $X = Y$, this set is the general linear group of $X$ nd will be denoted as $\GL(X)$. Given two operators $S, T \in \Emb(Y, X)$, we say that $Y$ and $X$ are \textit{homotopic} if there is a continuous mapping:
$$ \begin{array}{rcl}
[0, 1] & \to & \Emb(Y, X) \\
t & \mapsto & T_t
\end{array}$$
such that $T_0 = S$ and $T_1 = T$. Homotopy is an equivalence relation whose classes are exactly the connected components of $\Emb(Y, X)$. In \cite{Maurey}, the following result is mentionned as an exercise, in the special case of complex spaces:

\begin{thm}\label{thm:HomotopyC}

Let $X$ be an \gls{hi} spaces and $Y, Z$ to subspaces of $X$ that are isomorphic. Then there is an into isomorphism $T \colon Y \to X$ with $T(Y) = Z$, and such that $T$ is homotopic to $\incl_{Y, X}$ in $\Emb(Y, X)$.

\end{thm}

In this section, we give a proof of \prettyref{thm:HomotopyC} working as well in the complex as in the real case. This result can be seen as an analogue of \prettyref{thm:main} for subspaces: if two subspaces of an \gls{hi} space are isomorphic, then none of them can be ``too deep'' compared to the other. In particular, \prettyref{thm:main} is a consequence of \prettyref{thm:HomotopyC}: given $X$ an \gls{hi} space and $Y \subseteq X$ a proper subspace, if $Y$ is isomorphic to $X$, then $\Id_X$ should be homotopic to an isomorphism from $X$ to $Y$, which is impossible by continuity of the Fredholm index.

\smallskip

We actually prove a slightly stronger result than \prettyref{thm:HomotopyC}. Given $Y$ a Banach space, recall that an operator $R \colon Y \to Y$ is a \textit{reflection} of $Y$ if there exists $H \subseteq Y$ a hyperplane and $x_0 \in Y \setminus \{0\}$ such that $\forall x \in H \; R(x) = x$ and $R(x_0) = - x_0$. Call an \textit{antireflection} of $Y$ an operator of the form $-R$, where $R$ is a reflection of $Y$. Any two reflections $R_0$ and $R_1$ of the same space $Y$ are homotopic: indeed, given $x_0, x_1 \in Y \setminus \{0\}$ and $l_0, l_1 \in Y^* \setminus \{0\}$ such that for all $i$, $\ker(l_i) = \ker(R_i - \Id_Y)$ and $R_i(x_i) = -x_i$, then, it is not hard to find a continuous path between $(x_0, l_0)$ and $(\delta x_1, \varepsilon l_1)$ for some choice of signs $\delta, \epsilon \in \{-1, 1\}$ in the set $\{(x, l) \in X \times X^* \mid l(x) \neq 0\}$. This directly gives an homotopy between $R_0$ and $R_1$. In the same way, any two antireflections of the same space are homotopic. In what follows, for simplicity of notation, given $Y \subseteq X$ two Banach spaces, we will sometimes confound an automorphism $U$ of $Y$ with the into isomorphism $\incl_{Y, X} \circ U \colon Y \to X$. The theorem we prove is the following:

\begin{thm}\label{thm:HomotopyR}

Let $X$ be a Banach space, $Y$ be an \gls{hi} subspace of $X$, and $T\colon Y \to X$ be an into isomorphism. Then $T$ is either homotopic, in $\Emb(Y, X)$, to $\incl_{Y, X}$, or to $- \incl_{Y, X}$, or to all reflections of $Y$, or to all antireflections of $Y$. Moreover, this result can be refined in the following cases:

\begin{enumerate}

\item If $Y \neq X$, then $T$ is either homotopic to $\incl_{Y, X}$, or to $- \incl_{Y, X}$;

\item If $\K = \C$, then $T$ is homotopic to $\incl_{Y, X}$.

\end{enumerate}

\end{thm}

\prettyref{thm:HomotopyR} implies \prettyref{thm:HomotopyC}: indeed, if $U$ is any isomorphism $Y \to Z$, then \prettyref{thm:HomotopyR} shows that $U$ is homotopic, in $\Emb(Y, X)$, to an automorphism $V$ of $Y$, so $U \circ V^{-1}$ is an isomorphism $Y \to Z$ which is homotopic, in $\Emb(Y, X)$, to $\incl_{Y, X}$. Another consequence of \prettyref{thm:HomotopyR} is that if $Y$ is any \gls{hi} space and $X$ any Banach space, then $\Emb(Y, X)$ has:
\begin{itemize}
\setlength\itemsep{0pt}
\item at most four connected components if $\K = \R$ and $X$ is isomorphic to $Y$;
\item at most two connected components if $\K = \R$ and $X$ is not isomorphic to $Y$;
\item at most one connected component if $\K = \C$.
\end{itemize}

In what follows, we use the following notation. If $Y$ and $Z$ are (finite- or infinite-dimensional) subspaces of the same space $X$ that are in topological direct sum, and if $T \colon Y \to X$ and $U \colon Z \to X$ are two operators, then we denote by $T \oplus U$ the unique operator $Y \oplus Z \to X$ extending both $T$ and $U$.  The main ingredient in the proof of \prettyref{thm:HomotopyR} is the following lemma:

\begin{lem}\label{lem:homotopy}

Let $X$ be a Banach space, $Z \subseteq X$ be a subspace, and $F \subseteq X$ be a finite-dimensional subspace such that $Z \cap F = \{0\}$. Let $Y = Z \oplus F$, let $T \colon Y \to X$ be an into isomorphism, and suppose that $T_{\restriction Z}$ is homotopic to $\incl_{Z, X}$ in $\Emb(Z, X)$. Let $R$ be a reflection of $F$. Then, in $\Emb(Y, X)$, $T$ is either homotopic to $\incl_{Y, X}$, or to $\incl_{Z, X} \oplus R$.

\end{lem}


\begin{proof}[Proof of \prettyref{lem:homotopy}]

We first prove that the special case where $\dim(F) = 1$ implies the general case. This will be done by induction on $\dim(F)$. Suppose $\dim(F) \geqslant 2$ and let $R$ be a reflection of $F$. Let $G = \ker(R-\Id_F)$ and let $x_0 \in F \setminus \{0\}$ such that $R(x_0) = -x_0$, so we have $F = G \oplus \K x_0$. Let $R'$ be any reflection of $G$ and let $R'' = R' \oplus \incl_{\K x_0, F}$, which is a reflection of $F$. Applying the induction hypothesis to $Z$, $G$, the operator $T_{\restriction Z \oplus G}$ and the reflection $R'$, we obtain that $T_{\restriction Z \oplus G}$ is either homotopic to $\incl_{Z \oplus G, X}$, or to $\incl_{Z, X} \oplus R'$. So either $T_{\restriction Z \oplus G}$, or $T_{\restriction Z} \oplus \left(T_{\restriction G} \circ R'\right)$ is homotopic to $\incl_{Z \oplus G, X}$. Applying the one-dimensional case to $Z \oplus G$, $\K x_0$, respectively the operator $T_{\restriction Y}$ or the operator $T_{\restriction Z} \oplus \left(T_{\restriction F} \circ R''\right)$, and the reflection $-\Id_{\K x_0}$, we get that one of the operators $T_{\restriction Y}$ and $T_{\restriction Z} \oplus \left(T_{\restriction F} \circ R''\right)$  is homotopic to one of the operators  $\incl_{Y, X}$ and $\incl_{Y \oplus G, X} \oplus \left(-\incl_{\K x_0, X}\right) = \incl_{Z, X} \oplus R$. Summarizing, we get that the operator $T_{\restriction Y}$ is homotopic to one of the four following operators:
\begin{itemize}
\setlength\itemsep{0pt}
\item $\incl_{Y, X}$;
\item $\incl_{Z, X} \oplus R$;
\item $\incl_{Z, X} \oplus R''$;
\item $\incl_{Z, X} \oplus (R \circ R'')$.
\end{itemize}
But since the reflections $R$ and $R''$ are homotopic in $\Emb(F, F)$, we are actually in one of the first two cases, which finishes the first part of this proof.

\smallskip

We now prove the result in the special case where $\dim(F) = 1$. Let $\left(\widetilde{T}_t\right)_{t \in [0, 1]}$ be an homotopy between $\incl_{Z, X}$ and $T_{\restriction Z}$ in $\Emb(Z, X)$. By local constancy of the Fredholm index, for every $t \in [0, 1]$ we have $\widetilde{T}_t(Z) \neq X$. So by a standard compacity argument, we can find $0 = t_0 < t_1 < \ldots < t_n = 1$ and vectors $x_1, \ldots, x_n \in X$ such that for every $i < n$ and every $t \in [t_i, t_{i + 1}]$, we have $x_i \notin \widetilde{T}_t(Z)$. We can even assume that $F = \K x_0$ and that $x_{n-1} = T(x_0)$.

\smallskip

For every $t \in [0, 1]$ and every $x \in X$, we abusively denote by $\widetilde{T}_t \oplus x$ the unique operator $U \colon Y \to X$ such that $U_{\restriction Z} = \widetilde{T}_t$ and $U(x_0) = x$. Since $x_0 \mapsto - x_0$ is the only reflection of $F$, what we have to prove is that $T = \widetilde{T}_1 \oplus x_{n - 1}$ is homotopic to $\widetilde{T}_0 \oplus \varepsilon x_0$ in $\Emb(Y, X)$ for some sign $\varepsilon \in \{-1, 1\}$. To prove that, first observe that for every $i < n$ and for every $\varepsilon \in \{-1, 1\}$, $\widetilde{T}_{t_i} \oplus \varepsilon x_i$ and $\widetilde{T}_{t_{i+1}} \oplus \varepsilon x_i$ are homotopic in $\Emb(Y, X)$. So to conclude, it is enough to show that for every $1 \leqslant i < n$ and for every sign $\varepsilon \in \{-1, 1\}$, there exists a sign $\delta \in \{-1, 1\}$ such that $\widetilde{T}_{t_i} \oplus \varepsilon x_i$ is homotopic to $\widetilde{T}_{t_i} \oplus \delta x_{i-1}$ in $\Emb(Y, X)$.

\smallskip

To see this this, suppose that both segments $[\varepsilon x_i, x_{i-1}]$ and $[\varepsilon x_i, - x_{i-1}]$ intersect $\widetilde{T}_{t_i}(Z)$. This means that there exists $t, t' \in (0, 1)$ such that $t\varepsilon x_i + (1-t)x_{i-1} \in \widetilde{T}_{t_i}(Z)$ and $t'\varepsilon x_i - (1-t')x_{i-1} \in \widetilde{T}_{t_i}(Z)$. So $\left(\frac{t}{1-t} + \frac{t'}{1-t'}\right)\varepsilon x_i \in \widetilde{T}_{t_i}(Z)$, a contradiction. Thus, there exists $\delta \in \{-1, 1\}$ such that the segment $[\varepsilon x_i, \delta x_{i-1}]$ does not intersect $\widetilde{T}_{t_i}(Z)$. This means that the family $(\widetilde{T}_{t_i} \oplus (t\varepsilon x_i + (1-t)\delta x_{i-1}))$ is an homotopy between $\widetilde{T}_{t_i} \oplus \varepsilon x_i$ and $\widetilde{T}_{t_i} \oplus \delta x_{i-1}$ in $\Emb(Y, X)$, concluding the proof.

\end{proof}

\begin{proof}[Proof of \prettyref{thm:HomotopyR}]

We first prove the general case. \prettyref{lem:OneInfSing} show us that $T$ has at most one infinitely singular value. Replacing $T$ with $-T$ if necessary, we may assume that $T$ has no nonpositive infinitely singular value. This implies that for every $t \in [0, 1]$, the operator $tT + (1 - t)\incl_{Y, X}$ is finitely singular. For every $t \in [0, 1]$, we can find $Z_t$ a finite-codimensional subspace of $Y$ such that $(tT + (1 - t)\incl_{Y, X})_{\restriction Z_t}$ is an into isomorphism; there exists $U_t$ a neighborhood of $t$ in $[0, 1]$ such that for every $t' \in U_t$, $(t'T + (1 - t')\incl_{Y, X})_{\restriction Z_t}$ is still an into isomorphism. Select $t_1, \ldots t_n \in [0, 1]$ such that $\bigcup_{i=1}^n U_{t_i} = [0, 1]$, and let $Z = \bigcap_{i = 1}^n Z_{t_i}$. Then $Z$ is a finite-codimensional subspace of $Y$ such that for every $t \in [0, 1]$, $(tT + (1 - t)\incl_{Y, X})_{\restriction Z}$ is an into isomorphism. Thus, $T_{\restriction Z}$ and $\incl_{Z, X}$ are homotopic in $\Emb(Z, X)$.

\smallskip

Taking $F$ any complement of $Z$ in $Y$ and $R$ any reflection of $Z$, we can apply \prettyref{lem:homotopy} to $Z$, $F$, $T$ and $R$, showing that $T$ is either homotopic to $\incl_{Y, X}$, or to $\incl_{Z, X} \oplus R$ in $\Emb(Y, X)$. This concludes the general case, since $\incl_{Z, X} \oplus R$ is a reflection of $Y$.

\smallskip

To prove the case where $Y \neq X$, it is enough to show that any reflection of $Y$ is homotopic to $\incl_{Y, X}$ in $\Emb(Y, X)$. Given $R$ such a reflection, consider $x_0 \notin Y$ any vector, let $Z = Y \oplus \K x_0$ and $R' = R \oplus \incl_{\K x_0, Z}$. This defines a reflection of $Z$. This reflection is homotopic to the reflection $\incl_{Y, Z} \oplus \left(- \incl_{\K x_0, Z}\right)$ of $Z$ in $\Emb(Z, X)$, and this homotopy restricts into an homotopy between $R$ and $\incl_{Y, X}$ is $\Emb(Y, X)$.

\smallskip

To prove the case where $\K = \C$, it is enough to prove that for any complex space $X$, and for any decomposition $X = Y \oplus Z$ (where $Y$ and $Z$ can be finite- or infinite-dimensional), the operator $\incl_{Y, X} \oplus \left(- \incl_{Z, X}\right)$ is homotopic to $\Id_X$. An homotopy is given by $\left(\incl_{Y, X} \oplus \left(e^{i \pi t} \incl_{Z, X}\right)\right)_{t \in [0, 1]}$.

\end{proof}

\begin{rem}
\prettyref{thm:HomotopyR} is optimal, in the sense that there exists a real \gls{hi} space $X$ such that for every reflection $R$ of $X$, the operators $\Id_X$, $-\Id_x$, $R$ and $-R$ are pairwise non-homotopic in $\Emb(X, X)$, and for every subspace $Y$ of $X$, the operators $\incl_{Y, X}$ and $- \incl_{Y, X}$ are non-homotopic in $\Emb(Y, X)$. In particular, there exists a real Banach space whose general linear group has exactly four connected components. We sketch the proof of these facts below.

\smallskip

Take for $X$ the real space built by Gowers and Maurey in \cite{GowersMaureyHI}. It has the following property: for every subspace $Y$ of $X$, every operator $T \colon Y \to X$ has the form $\lambda \Id + S$, where $\lambda \in \R$ and $S$ is strictly singular. It is not hard to see that $\lambda$ is continuous in $T$ (use for example a similar method as in the proof of \prettyref{lem:OneInfSing}), thus showing that $\incl_{Y, X}$ and $- \incl_{Y, X}$ cannot be homotopic in $\Emb(Y, X)$.

\smallskip

We now prove that $\Id_X$ cannot be homotopic to a reflection of $X$ in $\Emb(X, X)$. For a real space Banach $Y$ and an operator $T \colon Y \to Y$, we denote its complexification by $T^\C \colon Y^\C \to Y^\C$. Denote by $\sigma(T)$ the spectrum of an operator $T$ on a complex space. Recall the following facts:

\begin{enumerate}
    \item If $T$ is an operator on a real space, then the spectrum of $T_\C$ is invariant under conjugation, and if $\lambda$ is an eigenvalue of $T_\C$, then $\bar \lambda$ is also one, with the same multiplicity;
    
    \item The spectrum of a strictly singular operator on a complex space consists in $0$ together with a either finitely many nonzero eigenvalues with finite multiplicity, or a sequence of nonzero eigenvalues with finite multiplicity converging to $0$ (see \cite{Maurey}, Proposition 6.1);
    
    \item If $T \colon Y \to Y$ is an operator on a complex space, and $V$ an open subset of $\C$ such that $\sigma(T) \cap V$ consists in finitely many eigenvalues of $T$ with finite multiplicity, then for every $S \in \cL(Y)$ close enough to $T$, $\sigma(S) \cap V$ consists in finitely many eigenvalues of $S$ and the sum of their multiplicities is equal to the sum of the multiplicities of the eigenvalues of $T$ in $V$ (see \cite{Kato}, Chapter Four, subsection 3.5).
    
\end{enumerate}

Now if there exists an homotopy $(T_t)_{t \in [0, 1]}$ in $\Emb(X, X)$ between $T_0 = \Id_X$ and $T_1 = R$ a reflection, then without loss of generality we can assume that for every $t \in [0, 1]$, $T_t = \Id_X + S_t$, where $S_t$ is strictly singular. In particular, $0 \notin \sigma({T_t}^\C)$, $\sigma({T_t}^\C)$ contains only finitely many negative elements and all of these elements are eigenvalues with finite multiplicity. Now if $t \in [0, 1]$ and $V$ is an open subset of $\C$ such that $V \cap \R = (- \infty, 0)$ and such that $\sigma({T_t}^\C) \cap V$ is exactly the set of negative eigenvalues of ${T_t}^\C$, then for $s$ in a neigborhood of $t$, the sum of the multiplicities of the eigenvalues of ${T_s}^\C$ that are in $V$ is equal to the sum of the multiplicities of the negative eigenvalues of ${T_t}^\C$. By invariance of the spectrum under conjugation, the sum of the multiplicities of ${T_s}^\C$ that are in $V \setminus (0, \infty)$ is even. Thus, the parity of the sum of the multiplicities of the negative eigenvalues of ${T_t}^\C$ is localy constant, so constant. This contradicts the fact that this sum is $0$ for $T_0$, and $1$ for $T_1$.

\end{rem}

\bigskip

\bibliographystyle{plain}
\bibliography{main}

  \par
  \bigskip
    \textsc{\footnotesize N. de Rancourt, Department of Logic, Faculty of Arts, Charles University, n\'am. Jana Palacha 2, 116 38 Praha 1, CZECH REPUBLIC}
    
    {\small \textit{Current adress:} Universit\"at Wien, Institut f\"ur Mathematik, Kurt G\"odel Research Center, Augasse 2-6, UZA 1 -- Building 2, 1090 Wien, AUSTRIA

    \textit{Email address:} \texttt{noe.de.rancourt@univie.ac.at}}

\end{document}